\documentclass{article} 
\usepackage{nips15submit_e,times}
\usepackage{hyperref}
\usepackage{url}
\usepackage{wrapfig}
\usepackage{mathrsfs}
\usepackage{epsfig}
\usepackage{graphics}
\usepackage{epsf}
\usepackage{amsmath, amsthm, amssymb, multirow, paralist}
\usepackage{url}
\usepackage{color}
\usepackage{algorithm,algorithmic}
\newtheorem{thm}{Theorem}

\newtheorem{lemma}{Lemma}

\newtheorem{assumption}[thm]{Assumption}

\def \Er {\mathcal{E}}
\def \bdelta {\boldsymbol{\delta}}
\def \S {\mathcal{S}}

\def \R {\mathbb{R}}

\def \Se {\mathcal{S}}

\def \x {\mathbf{x}}
\def \xh {\widehat{\mathbf{x}}}
\def \wh {\widehat{\mathbf{w}}}

\def \q {\mathbf{q}}

\def \c {\mathbf{c}}
\def \v {\mathbf{v}}
\def \w {\mathbf{w}}

\def \b {\mathbf{b}}
\def \e {\mathbf{e}}

\def \y {\mathbf{y}}
\def \u {\mathbf{u}}

\def \Xh {\widehat X}

\def \yh {\widehat{\mathbf y}}

\def \xt {\widetilde{\mathbf{x}}}
\def \et {\widetilde{\mathbf{e}}}

\title{Fast Sparse Least-Squares Regression with  Non-Asymptotic Guarantees}

\author{
Tianbao Yang$^*$, Lijun Zhang$^\dagger$, Qihang Lin$^*$, Rong Jin$^\ddagger$\\
$^*$The University of Iowa, $^\dagger$Nanjing University, $^\ddagger$Alibaba Group\\
\texttt{tianbao-yang@uiowa.edu, zhanglj@lamda.nju.edu.cn} \\
\texttt{qihang-lin@uiowa.edu, jinrong.jr@alibaba-inc.com}
}

\nipsfinalcopy 

\begin{document}

\maketitle

\begin{abstract}
In this paper, we study  a fast approximation method for {\it large-scale high-dimensional} sparse least-squares regression problem  by exploiting the Johnson-Lindenstrauss (JL) transforms, which embed a set of high-dimensional vectors into a low-dimensional space. In particular,  we propose to apply the JL transforms to the data matrix and the target vector  and then to solve a   sparse least-squares problem on the compressed data with a {\it slightly larger regularization parameter}. Theoretically, we establish the optimization error bound of the learned model for two different sparsity-inducing regularizers, i.e., the elastic net and the $\ell_1$ norm. Compared with previous relevant work, our analysis is {\it non-asymptotic and exhibits more insights} on the bound,  the sample complexity and  the regularization. As an illustration, we also provide an error bound of the {\it Dantzig selector} under JL transforms.    
\end{abstract}

\section{Introduction}
\vspace*{-0.1in}
Given a data matrix $X\in\R^{n\times d}$ with each row representing an instance~\footnote{$n$ is the number of instances  and $d$ is the number of features.} and a target vector $\y=(y_1,\ldots, y_n)^{\top}\in\R^n$, the sparse least-squares regression (SLSR) is to solve the following optimization problem:
\vspace*{-0.15in}
\begin{align}\label{eqn:lsr}
\w_* =\arg\min_{\w\in\R^d} \frac{1}{2n}\|X\w - \y\|_2^2  + \lambda R(\w)
\end{align}
where $R(\w)$ is a sparsity-inducing norm. In this paper, we consider two widely used sparsity-inducing norms: (i) the $\ell_1$ norm that leads to a formulation also known as LASSO~\cite{tibshirani96regression};  (ii) the mixture of $\ell_1$ and $\ell_2$ norm that leads to a formulation known as the Elastic Net~\cite{hastieElasticNet}. 
Although $\ell_1$ norm has been widely explored  and studied in SLSR, the elastic net usually yields better performance when there are highly correlated variables. Most previous studies on SLSR revolved around on two intertwined topics:  sparse recovery analysis and  efficient optimization  algorithms. We aim to present a fast approximation method for solving SLSR with a strong guarantee on the optimization error. 

Recent years have witnessed unprecedented growth in both the scale and the dimensionality of data.  As the size of data continues to grow, solving the problem~(\ref{eqn:lsr}) is still computationally difficult because (i) the memory limitations could lead to increased additional costs (e.g., I/O costs, communication costs in distributed environment); (ii) a large number $n$ of instances or a high dimension $d$ of features  usually implies a slow convergence of optimization (i.e., a large iteration complexity). {\bf In this paper, we study a fast approximation method that employes the JL transforms to reduce the size of $X\in\R^{n\times d}$ and $\y\in\R^n$}. In particular, let $A\in\R^{m\times n} (m\ll n)$ denote a linear transformation that obeys the JL lemma (c.f. Lemma~\ref{lem:jl}), we transform the data matrix and the target vector into $\Xh = AX\in\R^{m\times d}$ and $\yh=A\y\in\R^m$. Then we optimize a {\bf slightly modified} SLSR problem using the compressed data $\Xh$ and $\yh$ to obtain an approximate solution $\wh_*$. The proposed method is supported by  (i)  a theoretical  analysis that provides a strong guarantee of the proposed approximation method on the optimization error of $\wh_*$ in both $\ell_2$ norm and $\ell_1$ norm, i.e., $\|\wh_* - \w_*\|_2$ and $\|\wh_* - \w_*\|_1$; and (ii) empirical studies  on  a synthetic data and a real dataset. 
We emphasize that besides in large-scale learning, the  approximation method by JL transforms  can be also used in privacy concerned applications, which is beyond the scope of this work. 

In fact, our work is not the first that employes random reduction techniques to reduce the size of the data for SLSR and studies the theoretical guarantee of the approximate solution. The most relevant work is presented by Zhou \& Lafferty \& Wasserman~\cite{DBLP:conf/nips/ZhouLW07} (referred to as Zhou's work). Below we highlight several key differences from Zhou's work, which also emphasize our contributions: 
\begin{itemize}
\vspace*{-0.1in}
\item Our formulation on the compressed data is different from that in Zhou's work, which simply solves the same SLSR problem using the compressed data. We introduce a slightly larger $\ell_1$ norm regularizer, which enjoys an intuitive geometric explanation. As a result, it also sheds lights on the Dantzig selector~\cite{CT07} under JL transforms, a theoretical result of which is also presented. 

\item Zhou's work focused on the $\ell_1$ regularized least-squares regression and the Gaussian random projection. We consider two sparsity-inducing regularizers including the elastic net and the $\ell_1$ norm. Since our analysis is based on the JL lemma, hence any JL transforms are applicable. 

\item Zhou's theoretical analysis is {\it asymptotic}, which only holds when the number of instances $n$ approaches infinity, and  it requires strong assumptions about the data matrix and other parameters for obtaining sparsitency (i.e., the recovery of the support set) and the persistency (i.e., the generalization performance).  In contrast, our analysis of the optimization error {\it relies  on relaxed assumptions and is non-asymptotic}. In particular, for the $\ell_1$ norm we assume the standard  restricted eigen-value condition in sparse recovery analysis. For the elastic net, by exploring the strong convexity of the regularizer, we can  be even exempted from the restricted eigen-value condition and can derive better bounds when the condition is true. 
\end{itemize}
\vspace*{-0.1in}
The remainder of the paper is organized as follows. In Section~\ref{sec:rw}, we review some related work. We present the proposed method and main results in Section~\ref{sec:main} and~\ref{sec:dant}. Numerical experiments will be presented in Section~\ref{sec:exp} followed by conclusions. 

\section{Related Work}\label{sec:rw}
\vspace*{-0.1in}
\paragraph{Sparse Recovery Analysis.} 
The LASSO problem has been one of  the core problems in statistics and machine learning, which is essentially to learn a high-dimensional sparse vector $\u_*\in\R^d$ from (potentially noise) linear measurements $\y = X\u_* + \xi\in\R^n$. A rich theoretical literature~\cite{tibshirani96regression,zhao-2006-model,journals/tit/Wainwright09a} describes the consistency, in particular the sign consistency, of various sparse  regression techniques. A stringent ``irrepresentable condition'' has been established  to achieve sign consistency. To circumvent the stringent assumption, several studies~\cite{Jia:arXiv1208.5584,paul2008} have proposed to precondition the data matrix $X$ and/or the target vector $\y$ by $PX$ and $P\y$ before solving the LASSO problem, where $P$ is usually a $n\times n$ matrix. 
The oracle inequalities  of the solution to LASSO~\cite{Bickel09simultaneousanalysis} and  other sparse estimators (e.g., the Dantzig selector~\cite{CT07}) have also been established under restricted eigen-value conditions of the data matrix $X$ and the Gaussian noise assumption of $\xi$. The focus in these studies is on when the number of measurements $n$ is much less than the number of features, i.e., $n\ll d$. {\it Different from these work, we consider that both $n$ and $d$ are significantly large~\footnote{This setting recently receives increasing interest~\cite{DBLP:conf/nips/YenLLRD14}.} and aim to derive fast algorithms for solving the SLSR problem approximately by exploiting  the JL transforms. The recovery analysis is centered on the optimization error of the learned model with respect to the optimal solution $\w_*$ to~(\ref{eqn:lsr}), which together with the oracle inequality  of $\w_*$ automatically leads to an oracle inequality of the learned model under the Gaussian noise assumption. }

\vspace*{-0.1in}
\paragraph{Approximate Least-squares Regression.}
In numerical linear algebra, one important problem is the over-constrained least-squares problem, i.e., finding a vector $\w_{opt}$ such that the Euclidean norm of the residual error $\|X\w - \y\|_2$ is minimized,
where the data matrix $X\in\R^{n\times d}$ has $n\gg d$. The exact solver takes $O(nd^2)$ time complexity. Several pieces of works have proposed randomized algorithms for finding an approximate solution to the above problem in $o(nd^2)$~\cite{Drineas:2011:FLS:1936922.1936925,conf/soda/DrineasMM06}.  These works share the same paradigm by applying an appropriate random matrix $A\in\R^{m\times n}$ to both $X$ and $\y$ and solving the induced subproblem, i.e.,  $\wh_{opt}=\arg\min_{\w\in\R^d}\|A(X\w - \y)\|_2$. 
Relative-error bounds for $\|\y - X\wh_{opt}\|_2$ and $\|\w_{opt}-\wh_{opt}\|_2$ have been developed. {\it Although the proposed method uses a similar idea to reduce the size of the data, there is a striking difference between our work and these studies  in that  we consider the sparse regularized least-squares problem when both $n$ and $d$ are very large.} As a consequence, the analysis and the required condition on $m$ are substantially different. The analysis for over-constrained least-squares relies on the low-rank of the data matrix $X$, while our analysis hinges on the inherent sparsity of the optimal solution $\w_*$. In terms of the value of $m$ for accurate recovery,  approximate least-squares regression requires  $m=O(d\log d/\epsilon^2)$. In contrast, for the proposed method, our analysis exhibits that the order of $m$ is $O(s\log d/\epsilon^2)$, where $s$ is the sparsity of the optimal solution $\w_*$ to~(\ref{eqn:lsr}). In addition, the proposed method can utilize any JL transforms as long as they obey the JL lemma. Therefore, our method can benefit from recent advances in sparser JL transforms, leading to a fast  transformation of the data. 

\vspace*{-0.1in}
\paragraph{Random Projection based Learning.}
Random projection has been employed for addressing the computational challenge of high-dimensional learning problems~\cite{ML06:Balcan}. In particular, if let $\x_1,\ldots,\x_n\in\R^d$ denote a set of instances, by random projection we can reduce the high-dimensional features into a low dimensional feature space by $\xh_i=A\x_i\in\R^{m}$, where $A\in\R^{m\times d}$ is a random projection matrix. Several works have studied some theoretical properties of learning in the low dimensional space. For example, \cite{RP:SVM} considered the following problem and its reduced counterpart (R):
\vspace*{-0.12in}
\begin{align*}
&\w_*=\arg \min_{\w\in\R^d}\frac{1}{n}\sum_{i=1}^n\ell(\w^{\top}\x_i, y_i) + \frac{\lambda}{2}\|\w\|_2^2,\quad\text{R: }\min_{\u\in\R^m}\frac{1}{n}\sum_{i=1}^n\ell(\u^{\top}\xh_i, y_i) + \frac{\lambda}{2}\|\u\|_2^2
\end{align*} 
Paul et al.~\cite{RP:SVM} focused on SVM and showed that the margin and minimum enclosing ball in the reduced feature space are preserved to within a small relative error provided that the data matrix $X\in\R^{n\times d}$ is of low-rank. Zhang et al.~\cite{DBLP:conf/colt/ZhangMJYZ13} studied the problem of recovering the original optimal solution $\w_*$ and proposed a dual recovery approach, i.e., using the learned dual variable in the reduced feature space to recover the model in the original feature space. They also established a recovery error under the low-rank assumption of the data matrix. Recently, the low-rank assumption is alleviated by the sparsity assumption. Zhang et al.~\cite{DBLP:journals/tit/0005MJYZ14}  considered a case when the optimal solution $\w_*$ is sparse and Yang et al.~\cite{DBLP:journals/corr/Yang0JZ15} assumed the optimal dual solution is sparse and proposed to solve a $\ell_1$ regularized dual formulation using the reduced data. They both established a recovery error in the order of $O(\sqrt{s/m}\|\w_*\|_2)$, where $s$ is the sparsity of the optimal primal solution or the optimal dual solution. Random projection for feature reduction has also been applied to the ridge regression problem~\cite{DBLP:conf/nips/MaillardM09}.  {\it However, these methods  do not apply to the SLSR problem and their analysis is developed mainly for the $\ell_2$ norm square regularizer.} In order to maintain the sparsity of $\w$, we consider compressing the data instead of the features so that the sparse regularizer is maintained for encouraging sparsity. Moreover, our analysis exhibits an recovery error in the order of $O(\sqrt{s/m}\|\e\|_2)$, where $\e=X\w_* - \y$ whose magnitude could be much smaller than $\w_*$.

\vspace*{-0.1in}
\paragraph{The JL Transforms.}
The JL transforms refer to a class of transforms that obey the JL lemma~\cite{citeulike:7030987}, which states that any $N$ points in Euclidean space can be embedded into $O(\epsilon^2\log N)$ dimensions so that all pairwise Euclidean distances are preserved upto $1 \pm \epsilon$. Since the original Johnson-Lindenstrauss result, many transforms have been designed  to satisfy  the JL lemma, including Gaussian random matrices~\cite{dasgupta-2003-jl}, sub-Gaussian random matrices~\cite{Achlioptas:2003:DRP:861182.861189}, randomized Hadamard transform~\cite{Ailon:2006:ANN:1132516.1132597}, sparse JL transforms by random hashing~\cite{DBLP:conf/stoc/DasguptaKS10,Kane:2014:SJT:2578041.2559902}. The analysis presented in this work builds upon  the JL lemma and therefore our method can enjoy the computational benefits  of sparse JL transforms including  less memory and fast computation. 

\section{A Fast Sparse Least-Squares Regression}\label{sec:main}
\vspace*{-0.1in}
\paragraph{Notations:} Let $(\x_i, y_i), i=1,\ldots, n$ be a set of $n$ training instances, where $\x_i\in\R^d$ and $y_i\in\R$. 
We refer to $X=(\x_1,\x_2,\ldots, \x_n)^{\top}=(\bar\x_{1},\ldots, \bar\x_d)\in\R^{n\times d}$ as the data matrix and to $\y=(y_1,\ldots, y_n)^{\top}\in\R^n$ as the target vector, where $\bar\x_j$ denotes the $j$ column of $X$. To facilitate our analysis, let $R$ be the upper bound of $\max_{1\leq j\leq d}\|\bar\x_j\|_2\leq R$. Denote by $\|\cdot\|_1$ and $\|\cdot\|_2$ the $\ell_1$ norm and the $\ell_2$ norm of a vector.  A function $f(\w):\R^d\rightarrow\R$ is $\lambda$-strongly convex with respect to $\|\cdot\|_2$ if  $\forall \w, \u\in\R^d$ it satisfies 
$f(\w) \geq f(\u) + \partial f(\u)^{\top}(\w - \u) + \frac{\lambda}{2}\|\w-\u\|_2^2$. 
A function $f(\w)$ is $L$-smooth with respect to $\|\cdot\|_2$ if for $\forall\w,\u\in\R^d$, 
$\|\nabla f(\w) - \nabla f(\w)\|_2\leq L\|\w - \u\|_2$, 
where $\partial f(\cdot)$ and $\nabla f(\cdot)$ denotes the sub-gradient and the gradient, respectively. In the analysis below for the LASSO problem, we will use the following restricted eigen-value condition~\cite{Bickel09simultaneousanalysis}. 
\begin{assumption}\label{ass:rc}
For any integer $1\leq s\leq d$, the matrix $X$ satisfies the restricted eigen-value condition at the sparsity level $s$ if there exist positive constants $\phi_{\min}(s)$ and $\phi_{\max}(s)$ such that 
\vspace*{-0.02in}
\begin{align*}
\phi_{\min}(s)=\min_{\w\in\R^d, 1\leq\|\w\|_0\leq s}\frac{\frac{1}{n}\w^{\top}X^{\top}X\w}{\|\w\|_2^2},\quad \text{and}\quad\quad \phi_{\max}(s)=\max_{\w\in\R^d, 1\leq\|\w\|_0\leq s}\frac{\frac{1}{n}\w^{\top}X^{\top}X\w}{\|\w\|_2^2}
\end{align*}
\end{assumption}
The goal of SLSR is to learn an optimal vector $\w_*=(w_{*1}, \ldots, w_{*d})^{\top}$ that minimizes the sum of the least-squares error and a sparsity-inducing regularizer.  We consider two different sparsity-inducing regularizers: (i) the $\ell_1$ norm: $R(\w) = \|\w\|_1 = \sum_{i=1}^d|w_i|$; (ii) the elastic net: $R(\w) = \frac{1}{2}\|\w\|_2^2 + \frac{\tau}{\lambda}\|\w\|_1$. 
Thus, we rewrite the problem in~(\ref{eqn:lsr}) into the following  form: 
\vspace*{-0.02in}
\begin{align}\label{eqn:lsr-2}
\w_* =\arg\min_{\w\in\R^d} \frac{1}{2n}\|X\w - \y\|_2^2  + \frac{\lambda}{2}\|\w\|_2^2 + \tau \|\w\|_{1}
\end{align}
When $\lambda=0$ the problem is the LASSO problem and when $\lambda>0$ the problem is the Elastic Net problem. 
Although many optimization algorithms have been developed for solving~(\ref{eqn:lsr-2}), they could still suffer from high computational complexities for large-scale high-dimensional data due to  (i) an $O(nd)$ memory complexity  and (ii) an $\Omega(nd)$ iteration complexity.  

To alleviate the two complexities, we consider using the JL transforms to reduce the size of data, which are discussed in more details in subsection~\ref{sec:jl}. In particular, we let $A\in\R^{m\times n}$ denote the transformation matrix corresponding to a JL transform, then we compute a compressed data by $\Xh = AX\in\R^{m\times d}$ and $\yh = A\y\in\R^m$, and then solve the following problem: 
\vspace*{-0.02in}
\begin{align}\label{eqn:lsr-r}
\wh_* =\arg\min_{\w\in\R^d} \frac{1}{2n}\|\Xh\w - \yh\|_2^2  + \frac{\lambda}{2}\|\w\|_2^2 + (\tau+\sigma) \|\w\|_{1}
\end{align}
where $\sigma>0$, whose theoretical value is exhibited later. We emphasize that to obtain a bound on the optimization error of $\wh_*$, i.e., $\|\wh_*-\w_*\|$,  it is important to increase the value of the regularization parameter before the $\ell_1$ norm. Intuitively, after compressing the data the optimal solution may become less sparse, hence increasing the regularization parameter can pull the solution towards closer to the original optimal solution. 

\textbf{Geometric Interpretation.} We  can also explain the added parameter $\sigma$ from a {\it geometric viewpoint}, which sheds insights on the theoretical value of $\sigma$ and the analysis for the Dantzig selector under JL transforms. Without loss of generality, we consider $\lambda=0$. Since $\w_*$ is the optimal solution to the original problem, then there exists a sub-gradient $g\in\partial \|\w_*\|_1$ such that 
$\frac{1}{n}X^{\top}(X\w_* - \y) + \tau g = 0$.  
Since $\|g\|_\infty\leq 1$, therefore $\w_*$ must satisfy 
$\frac{1}{n}\|X^{\top}(X\w_* - \y)\|_\infty\leq \tau$, 
which is also the constraint in the Dantzig selector. Similarly, the compressed problem~(\ref{eqn:lsr-r}) also defines a domain of the  optimal solution  $\wh_*$, i.e.,  
\vspace*{-0.1in}
\begin{align}\label{eqn:dome}
\widehat{\mathcal D}_{\w}=\left\{\w\in\R^d: \frac{1}{n}\|\Xh^{\top}(\Xh\w - \yh)\|_\infty\leq \tau + \sigma\right\}
\end{align}
It turns out that $\sigma$ is added to ensure that the original optimal solution $\w_*$ lies in $\widehat{\mathcal D}_{\w}$ provided that $\sigma$ is set appropriately, which can be verified as follows: 
\begin{align*}
&\frac{1}{n}\left\|\Xh^{\top}(\Xh\w_* - \yh)\right\|_\infty   = \frac{1}{n}\left\|X^{\top}(X\w_* - \y) + \Xh^{\top}(\Xh\w_* - \yh) -X^{\top}(X\w_* - \y)\right\|_\infty \\
&\leq \frac{1}{n}\|X^{\top}(X\w_* - \y)\|_\infty + \frac{1}{n}\left\|\Xh^{\top}(\Xh\w_* - \yh) -X^{\top}(X\w_* - \y)\right\|_\infty\\
& \leq \tau +  \frac{1}{n}\|X^{\top}(A^{\top}A-I)(X\w_* - \y) \|_\infty
\end{align*}
Hence, if we set $\sigma\geq \frac{1}{n}\|X^{\top}(A^{\top}A-I)(X\w_* - \y) \|_\infty$, it is guaranteed that $\w_*$  also lies in $\widehat{\mathcal D}_{\w}$. Lemma~\ref{lem:2} in subsection~\ref{sec:res} provides an upper bound $\frac{1}{n}\|X^{\top}(A^{\top}A-I)(X\w_* - \y) \|_\infty$, therefore exhibits a theoretical value of $\sigma$. The above explanation also sheds lights on the Dantzig selector under JL transforms as presented in Section~\ref{sec:dant}.

\subsection{Optimization}\label{sec:opt}
\vspace*{-0.1in}
Before presenting the theoretical guarantee of the obtained solution $\wh_*$, we compare the optimization of the original problem~(\ref{eqn:lsr-2}) and the compressed problem~(\ref{eqn:lsr-r}). In particular, we focus on $\lambda>0$ since the optimization of the problem with only $\ell_1$ norm can be completed by adding the $\ell_2$ norm square with a small value of $\lambda$~\cite{DBLP:conf/icml/Shalev-Shwartz014}. 

We choose the recently proposed accelerated stochastic proximal coordinate gradient method (APCG)~\cite{DBLP:conf/nips/LinLX14}. The reason are threefold: (i) it achieves  an accelerated convergence for optimizing~(\ref{eqn:lsr-2}), i.e., a linear convergence with a square root dependence on the condition number; (ii) it updates randomly selected coordinates of $\w$, which is well suited for solving~(\ref{eqn:lsr-r}) since the dimensionality $d$ is much larger than the equivalent number of examples $m$; (iii) it leads to a much simpler analysis of the condition number for the compressed problem~(\ref{eqn:lsr-r}).
First, we write the objective functions in~(\ref{eqn:lsr-2}) and~(\ref{eqn:lsr-r}) into the following general form:  
\begin{align}\label{eqn:general}
f(\w)  + \tau'\|\w\|_1 = \left(\frac{1}{2n}\|C\w - \b\|_2^2 + \frac{\lambda}{2}\|\w\|_2^2\right) + \tau'\|\w\|_1
\end{align}
where $C=(\c_1,\ldots, \c_d)\in\R^{N\times d}$. For simplicity, we consider the case when each block of coordinates corresponds to only one coordinate. The key assumption of APCG is that the function $f(\w)$ should be coordinate-wise smooth. To this end, we let $\e_j$ denote the $j$-th column of the identity matrix and note that 
\begin{align*}
\nabla f(\w) &= \frac{1}{n}C^{\top}C\w  - \frac{1}{n}C^{\top}\b+ \lambda \w,\quad \nabla_j f(\w) = \e_j^{\top}\nabla f(\w) = \frac{1}{n}\e_j^{\top}C^{\top}C\w + \lambda w_j - \frac{1}{n}[C^{\top}\b]_j
\end{align*}
Assume $\max_{1\leq j\leq d}\|\c_j\|_2\leq R_c$, then for any $h_j\in\R$, we have
\begin{align*}
|\nabla_jf(\w+h_j\e_j) - \nabla_j f(\w)|& = \left|\frac{1}{n}\e_j^{\top}C^{\top}C(\w + \e_jh_j)- \frac{1}{n}\e_j^{\top}C^{\top}C\w + \lambda h_j\right |\\
&\leq \left(\frac{1}{n}|\e_j^{\top}C^{\top}C\e_j| + \lambda\right)|h_j|\leq \left(\frac{R_c^2}{n} + \lambda\right)|h_j|
\end{align*}
Therefore $f(\w)$ is coordinate-wise smooth and the smooth parameter is $R_c^2/n + \lambda$. On the other hand $f(\w)$ is also $\lambda$-strongly convex function. Therefore the condition number that affects the iteration complexity is  $\kappa = \frac{R_c^2/n + \lambda}{\lambda}$, and the iteration complexity is given by 
\begin{align*}
O\left(d\sqrt{\kappa}\log(1/\epsilon_o)\right) = O\left(d\sqrt{\frac{R_c^2/n + \lambda}{\lambda}}\log(1/\epsilon_o)\right) = O\left(\left[d + d\sqrt{\frac{R_c^2}{n\lambda}}\right]\log(1/\epsilon_o)\right)
\end{align*}
where $\epsilon_o$ is an accuracy for optimization.  Since the per-iteration complexity of APCG for~(\ref{eqn:general}) is $O(N)$, therefore the time complexity is given by $\widetilde O\left(Nd + Nd\sqrt{\frac{R_c^2}{n\lambda}}\right)$, where $\widetilde O$ suppresses the logarithmic term. Next, we can analyze and compare the time complexity of optimization for~(\ref{eqn:lsr-2}) and~(\ref{eqn:lsr-r}).  For~(\ref{eqn:lsr-2}),  $N=n$ and $R_c=R$. For~(\ref{eqn:lsr-r}) $N=m$, and by the JL lemma for $A$ (Lemma~\ref{lem:jl}), with a high probability $1-\delta$ we  have 
$R_c = \max_{1\leq j\leq d}\|A\bar\x_j\|_2\leq \max_{1\leq j\leq d}\sqrt{1+\epsilon_m}\|\bar\x_j\|_2$, 
where $\epsilon_m = O(\sqrt{\log(d/\delta)/m})$. Let $m$ be sufficiently large, we can conclude that $R_c$ for $\Xh$ is $O(R)$. Therefore, the time complexities of APCG for solving~(\ref{eqn:lsr-2}) and~(\ref{eqn:lsr-r}) are
\begin{align*}
(2): O\left(\left[nd + dR\sqrt{\frac{n}{\lambda}}\right]\log(1/\epsilon_o)\right), \quad\quad (3): O\left(\frac{m}{n}\left[nd + dR\sqrt{\frac{n}{\lambda}}\right]\log(1/\epsilon_o)\right)
\end{align*}  
Hence, we can see that the optimization time complexity of APCG for solving~(\ref{eqn:lsr-r}) can be reduced upto  a factor of $1-\frac{m}{n}$, which is substantial when $m\ll n$. The total time complexity is discussed after we introduce the JL lemma. 

\subsection{JL Transforms and Running Time}\label{sec:jl}
\vspace*{-0.1in}
Since the proposed method builds on the JL transforms, we present a JL lemma and mention several JL transforms. 

\begin{lemma}\label{lem:jl}[JL Lemma~\cite{citeulike:7030987}]
For any integer $n > 0$, and any $0 < \epsilon,\delta < 1/2$, there exists a probability distribution on $m\times n$ real matrices $A$ such that there exists a small universal constant $c>0$ and for any fixed $\bar\x$ with a probability at least $1-\delta$, we have
\begin{align}\label{lem:JL}
\left|\|A\bar\x\|_2^2 - \|\bar\x\|^2_2\right|\leq c\sqrt{\frac{\log(1/\delta)}{m}}\|\bar\x\|_2^2
\end{align}
\end{lemma}
\vspace*{-0.1in}
In other words, in order to preserve the Euclidean norm  for any vector $\bar\x\in\{\bar\x_1,\ldots, \bar\x_d\}$ within a relative error $\epsilon$, we need to have $m=\Theta(\epsilon^{-2}\log(d/\delta))$.  Proofs of the JL lemma can be found in many studies (e.g.,~\cite{dasgupta-2003-jl,Achlioptas:2003:DRP:861182.861189,Ailon:2006:ANN:1132516.1132597,DBLP:conf/stoc/DasguptaKS10,Kane:2014:SJT:2578041.2559902}). The value of $m$ in the JL lemma is optimal~\cite{DBLP:journals/talg/JayramW13}. In these studies, different JL transforms $A\in\R^{m\times n}$ are also exhibited, including Gaussian random matrices~\cite{dasgupta-2003-jl}:, subGaussian random matrices~\cite{Achlioptas:2003:DRP:861182.861189}, randomized Hadamard transform~\cite{Ailon:2006:ANN:1132516.1132597} and sparse JL transforms~\cite{DBLP:conf/stoc/DasguptaKS10,Kane:2014:SJT:2578041.2559902}. For more discussions on these JL transforms, we refer the readers to~\cite{DBLP:journals/corr/Yang0JZ15}. 

\textbf{Transformation time  complexity and Total Amortizing time complexity.}
Among all the JL transforms mentioned above, the transform using the  Gaussian random matrices is the most expensive that takes $O(mnd)$ time complexity when applied to $X\in\R^{n\times d}$, while randomized Hadamard transform and sparse JL transforms can reduce it to $\widetilde O(nd)$ where $\widetilde O(\cdot)$ suppresses only a logarithmic factor. Although the transformation time complexity still scales as $nd$, the computational benefit of the JL transform can become more  prominent when we consider the amortizing time complexity. In particular, in machine learning, we usually need to tune the regularization parameters (aka cross-validation) to achieve a better generalization performance. Let $K$ denote the total number of times of solving~(\ref{eqn:lsr-2}) or~(\ref{eqn:lsr-r}), then the amortizing time complexity is given by $\text{time}_{proc} + K\cdot{time}_{opt}$, where $\text{time}_{proc}$ refers to the time of the transformation (zero for solving~(\ref{eqn:lsr-2})) and $\text{time}_{opt}$ is the optimization time.  Since $\text{time}_{opt}$ for~(\ref{eqn:lsr-r}) is reduced significantly, hence the total amortizing time complexity of the proposed method for SLSR is much reduced.  

\subsection{Theoretical Guarantees}\label{sec:res}
\vspace*{-0.1in}
Next, we present the theoretical guarantees on the optimization error of the obtained solution $\wh_*$.   {\it We emphasize that one can easily obtain  the oracle inequalities for $\wh_*$ using the optimization error and the oracle inequalities of $\w_*$~\cite{Bickel09simultaneousanalysis} under the Gaussian noise model, which are omitted here}. We use the  notation $\e$ to denote  $X\w_* - \y=\e$ and assume $\|\e\|_2\leq \eta$. Again, we denote by $R$ the upper bound of column vectors in $X$, i.e., $\max_{1\leq j\leq d}\|\bar\x_i\|_2\leq R$. We first present two technical lemmas. All proofs are included in the appendix. 

\begin{lemma}\label{lem:2}
Let 
$\displaystyle\q  = \frac{1}{n}X^{\top}(A^{\top}A -I )\e$. 
With a probability at least $1-\delta$, we have
\vspace*{-0.1in}
\[
\|\q\|_\infty\leq \frac{c\eta R}{n}\sqrt{\frac{\log(d/\delta)}{m}},
\]
where $c$ is the universal constant in the JL Lemma.
\end{lemma}
\begin{lemma}\label{lem:3}
Let 
$\displaystyle \rho(s) = \max_{\|\w\|_2\leq 1, \|\w\|_1\leq \sqrt{s}}\frac{1}{n}\left|\w^{\top}(X^{\top}X - \Xh^{\top}\Xh)\w\right|$.  
If $X$ satisfies the restricted eigen-value condition as in \textbf{Assumption}~\ref{ass:rc},  then with a probability at least $1-\delta$, we have
\[
\rho(s)\leq 16c\phi_{\max}(s)\sqrt{\frac{\log(1/\delta) + 2s\log(36d/s)}{m}},
\]
where $c$ is the universal constant in the JL lemma. 
\end{lemma}
{\bf Remark: } Lemma~\ref{lem:2} is used in the analysis for Elastic Net, LASSO and Dantzig selector. Lemma~\ref{lem:3} is used in the analysis for LASSO and Dantzig selector. 

\begin{thm}[Optimization Error for Elastic Net]\label{thm:net}
Let $\sigma = \Theta\left(\frac{\eta R}{n}\sqrt{\frac{\log(d/\delta)}{m}}\right)\geq\frac{2c\eta R}{n}\sqrt{\frac{\log(d/\delta)}{m}}$, where $c$ is an universal constant  in the JL lemma. Let $\w_*$ and $\wh_*$ be the optimal solutions to~(\ref{eqn:lsr-2}) and~(\ref{eqn:lsr-r}) for $\lambda>0$, respectively. 
Then with a probability at least $1-\delta$, for $p=1$ or $2$ we have 
\begin{align*}
\|\wh_* - \w_*\|_p\leq O\left(\frac{\eta R}{n\lambda}\sqrt{\frac{s^{2/p}\log(d/\delta)}{m}}\right).
\end{align*}
\end{thm}
{\bf Remark: } First, we can see that the value of $\sigma$ is large than $\|\q\|_\infty$ with a high probability due to Lemma~\ref{lem:2}, which is consistent with our geometric interpretation.   The upper bound of the optimization error exhibits several interesting properties: (i) the term of $\sqrt{\frac{s^{2/p}\log(d/\delta)}{m}}$ occurs commonly in theoretical results of sparse recovery~\cite{vladimir-2011-oracle}; (ii) the term of  $R/\lambda$ is related to the condition number of the optimization problem~(\ref{eqn:lsr-2}), which reflects the intrinsic difficulty of optimization; and (iii) the term of $\eta/n$ is related to the empirical error of the optimal solution $\w_*$. This term makes sense because  if $\eta=0$ indicating that the optimal solution $\w_*$ satisfies $X\w_* -\y =0$, then it is straightforward to verify that $\w_*$ also satisfies the optimality condition of~(\ref{eqn:lsr-2}) for $\sigma=0$. Due to the uniqueness of the optimal solution to~(\ref{eqn:lsr-2}), thus $\wh_* = \w_*$.


\begin{thm}[Optimization Error for LASSO]\label{thm:lasso} Assume $X$ satisfies the restricted eigen-value condition in \textbf{Assumption}~\ref{ass:rc}. 
Let $\sigma = \Theta\left(\frac{\eta R}{n}\sqrt{\frac{\log(d/\delta)}{m}}\right)\geq\frac{2c\eta R}{n}\sqrt{\frac{\log(d/\delta)}{m}}$, where $c$ is an universal constant  in the JL lemma. Let $\w_*$ and $\wh_*$ be the optimal solutions to~(\ref{eqn:lsr-2}) and~(\ref{eqn:lsr-r}) with $\lambda=0$, respectively, and  $\Lambda=\phi_{\min}(16s) - 2\rho(16s)$. 
Assume $\Lambda>0$, then with a probability at least $1-\delta$, for $p=1$ or $2$ we have
\vspace*{-0.15in}
\begin{align*}
\|\wh_* - \w_*\|_p\leq O\left(\frac{\eta R}{n\Lambda}\sqrt{\frac{s^{2/p}\log(d/\delta)}{m}}\right)
\end{align*}
\end{thm}
\vspace*{-0.1in}
{\bf Remark: } Note that $\lambda$  in Theorem~\ref{thm:net} is replaced by $\Lambda$ in Theorem~\ref{thm:lasso}. In order to make the result to be valid, we must have $\Lambda>0$, i.e., $m\geq \Omega(\kappa^2(16s)(\log(1/\delta) + 2s\log(36d/s)))$, where $\kappa(16s)=\frac{\phi_{\max}(16s)}{\phi_{\min}(16s)}$. In addition, if the conditions in Theorem~\ref{thm:lasso} hold, the result in Theorem~\ref{thm:net} can be made stronger by replacing $\lambda$ with $\lambda +\Lambda$. 

\section{Dantzig Selector under JL transforms}\label{sec:dant}
\vspace*{-0.1in}
In light of our geometric explanation of $\sigma$, we present the Dantzig selector under JL transforms and its theoretical guarantee. The original Dantzig selector is  the optimal solution to the following problem: 
\vspace*{-0.15in}
\begin{align}\label{eqn:dan}
\w^D_* = \min_{\w\in\R^d}\|\w\|_1, \quad\quad \text{s.t.}\quad\frac{1}{n}\|X^{\top}(X\w - \y)\|_\infty\leq \tau
\end{align}
Under JL transforms, we propose the following estimator
\begin{align}\label{eqn:danr}
\wh^D_* = \min_{\w\in\R^d}\|\w\|_1, \quad\quad \text{s.t.}\quad\frac{1}{n}\left\|\Xh^{\top}(\Xh\w - \yh)\right\|_\infty\leq \tau + \sigma
\end{align}
From previous analysis, we show that $\w_*^D$ satisfies the constraint in~(\ref{eqn:danr}) provided that $\sigma \geq \|\q\|_\infty$, which is the key to establish the following result. 
\begin{thm}[Optimization Error for Dantzig Selector]\label{thm:dan}  Assume $X$ satisfies the restricted eigen-value condition in \textbf{Assumption}~\ref{ass:rc}. Let $\sigma = \Theta\left(\frac{\eta R}{n}\sqrt{\frac{\log(d/\delta)}{m}}\right)\geq\frac{c\eta R}{n}\sqrt{\frac{\log(d/\delta)}{m}}$, where $c$ is an universal constant  in the JL lemma. Let $\w^D_*$ and $\wh^D_*$ be the optimal solutions to~(\ref{eqn:dan}) and~(\ref{eqn:danr}), respectively, and  $\Lambda = \phi_{\min}(4s) - \rho(4s)$. 
Assume $\Lambda>0$, then with a probability at least $1-\delta$, for $p=1$ or $2$ we have
\vspace*{-0.1in}
\begin{align*}
\|\wh^D_* - \w^D_*\|_p\leq O\left(\frac{\eta R}{n\Lambda}\sqrt{\frac{s^{2/p}\log(d/\delta)}{m}} + \frac{\tau s^{1/p}}{\Lambda}\right)
\end{align*}
\end{thm}
\vspace*{-0.1in}
{\bf Remark:} Compared to the result in Theorem~\ref{thm:lasso}, the definition of $\Lambda$ is slightly different, and there is an additional term of $\frac{\tau s^{1/p}}{\Lambda}$. This additional term seems unavoidable since $\eta=0$ doest not necessarily indicate $\w_*^D$ is also the optimal solution to~(\ref{eqn:danr}). However, this should not be a concern  if we consider the oracle inequality of $\wh_*^D$ via the oracle inequality of $\w_*^D$, which is $\|\w_*^D - \u_*\|_{p}\leq O\left(\frac{\tau s^{1/p}}{\phi_{\min}(4s)}\right)$  under  the Gaussian noise assumption and  $\tau = \Theta\left(\sqrt{\frac{\log d}{n}}\right)$.

\begin{figure}[t]
\center
 \includegraphics[width=0.3\textwidth]{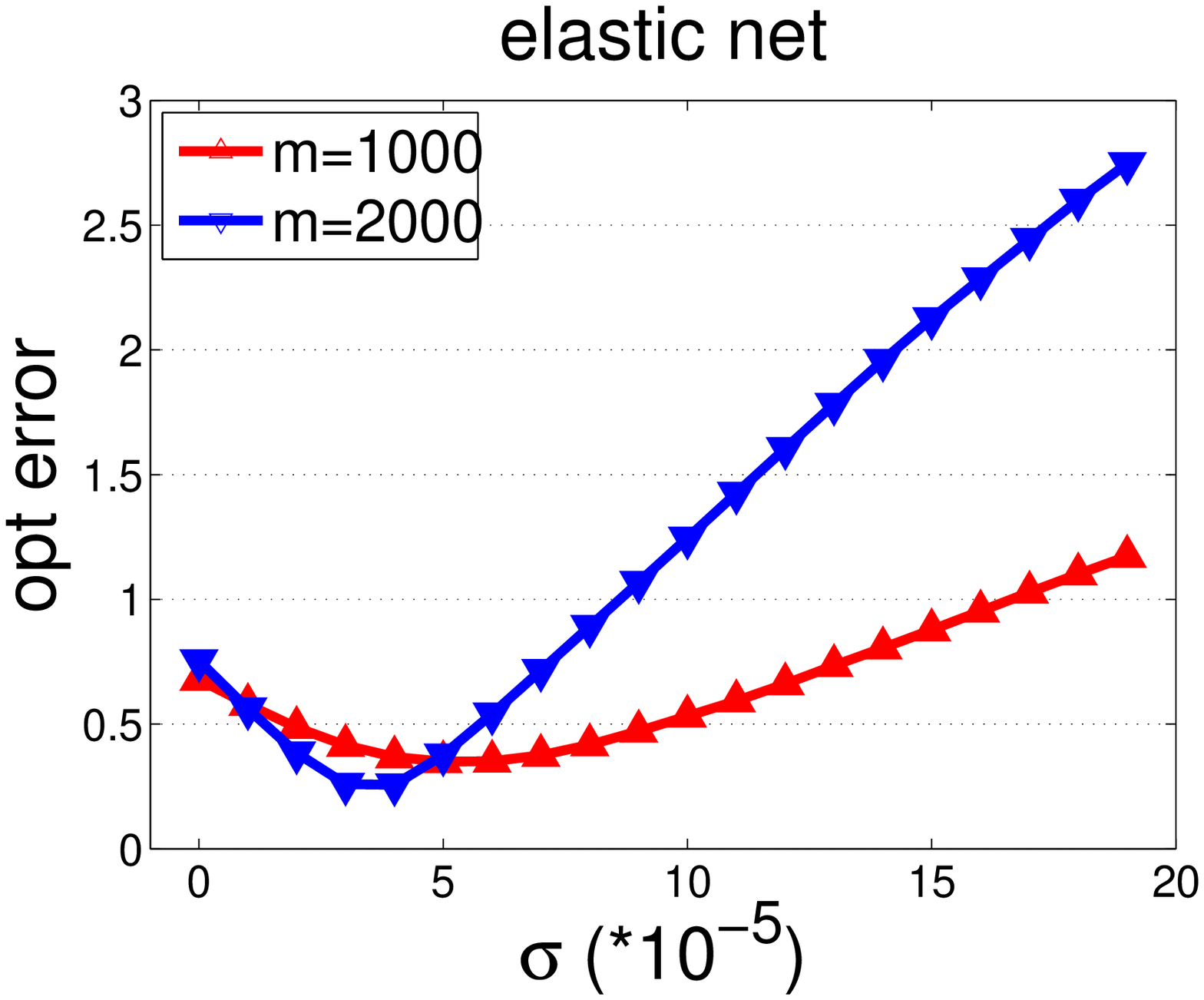}\hspace*{-0.1in}
 \includegraphics[width=0.3\textwidth]{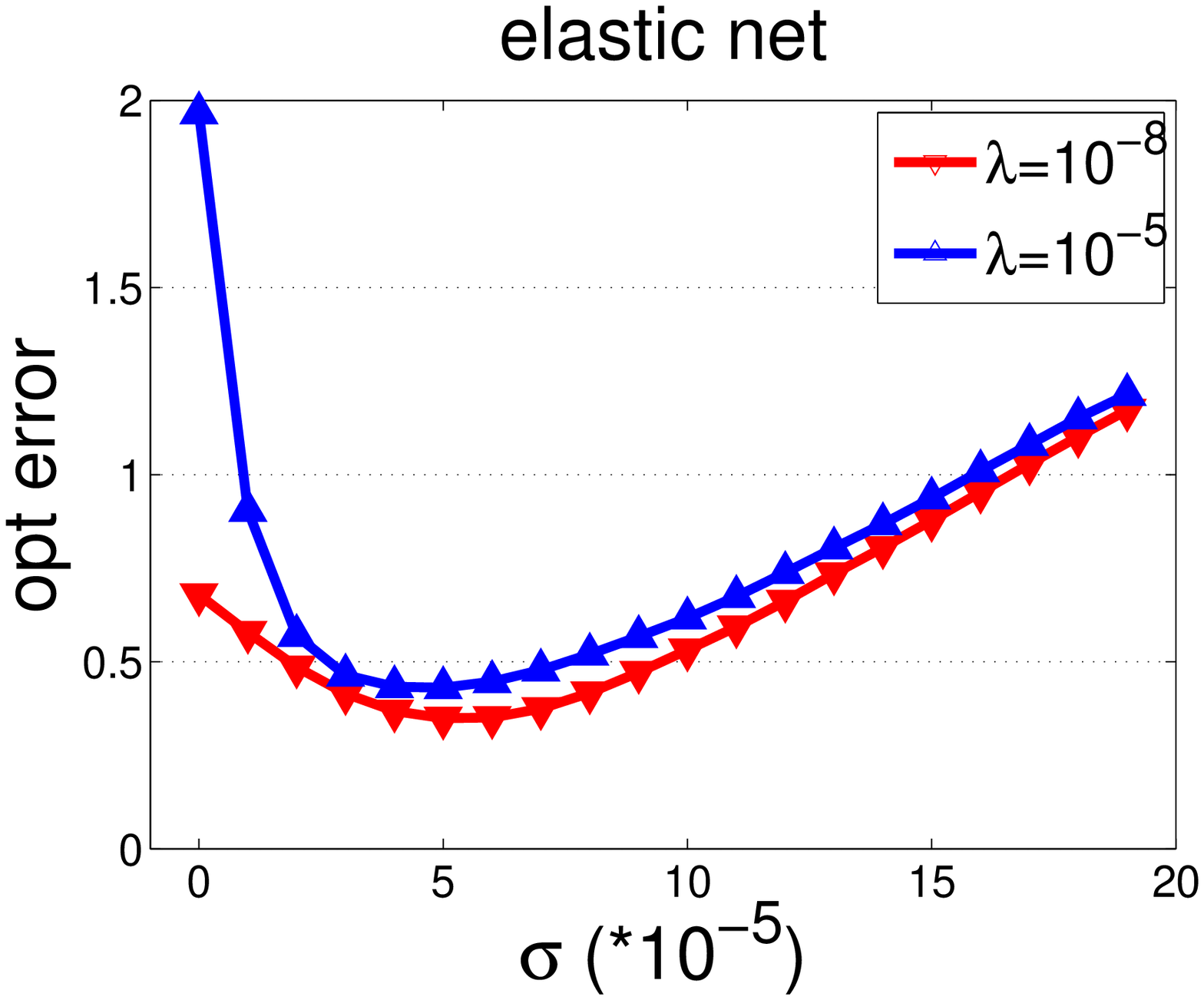}\hspace*{-0.1in}
 \includegraphics[width=0.3\textwidth]{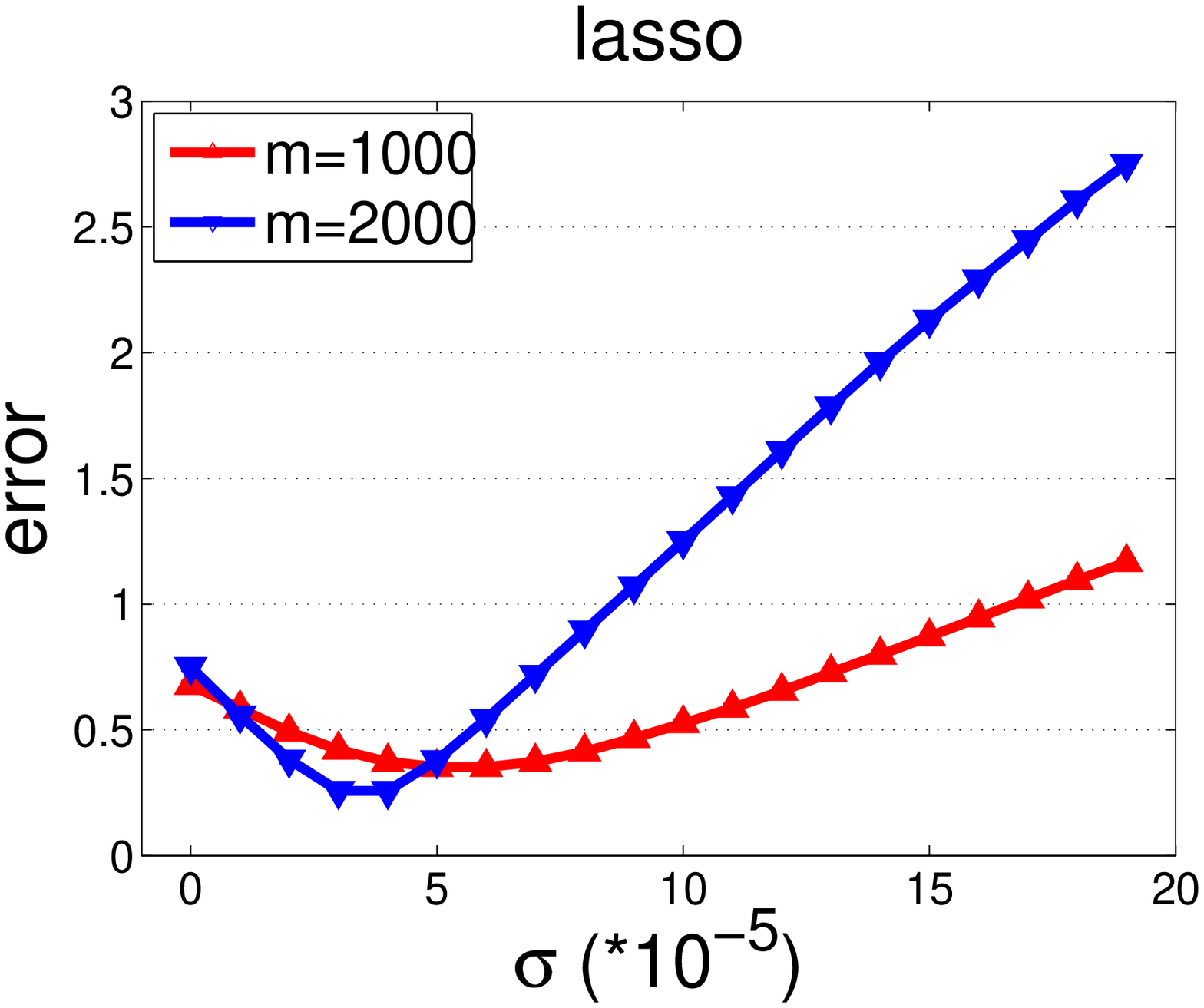}\hspace*{-0.1in}
 \vspace*{-0.1in}
\caption{Optimization error of elastic net and lasso under different settings on the synthetic data.}
\label{fig:syn}
\center
 \includegraphics[width=0.3\textwidth]{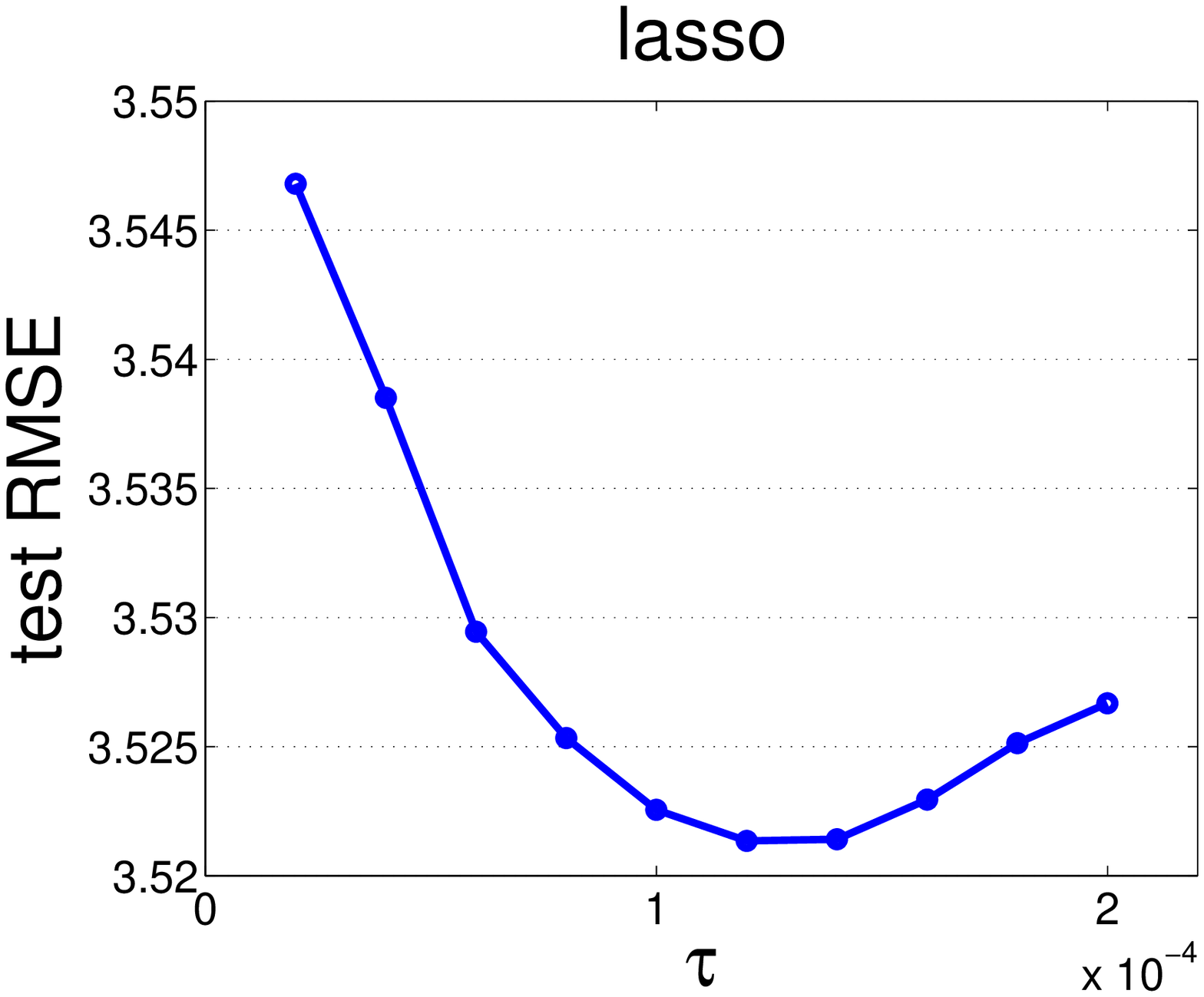}\hspace*{-0.1in}
 \includegraphics[width=0.3\textwidth]{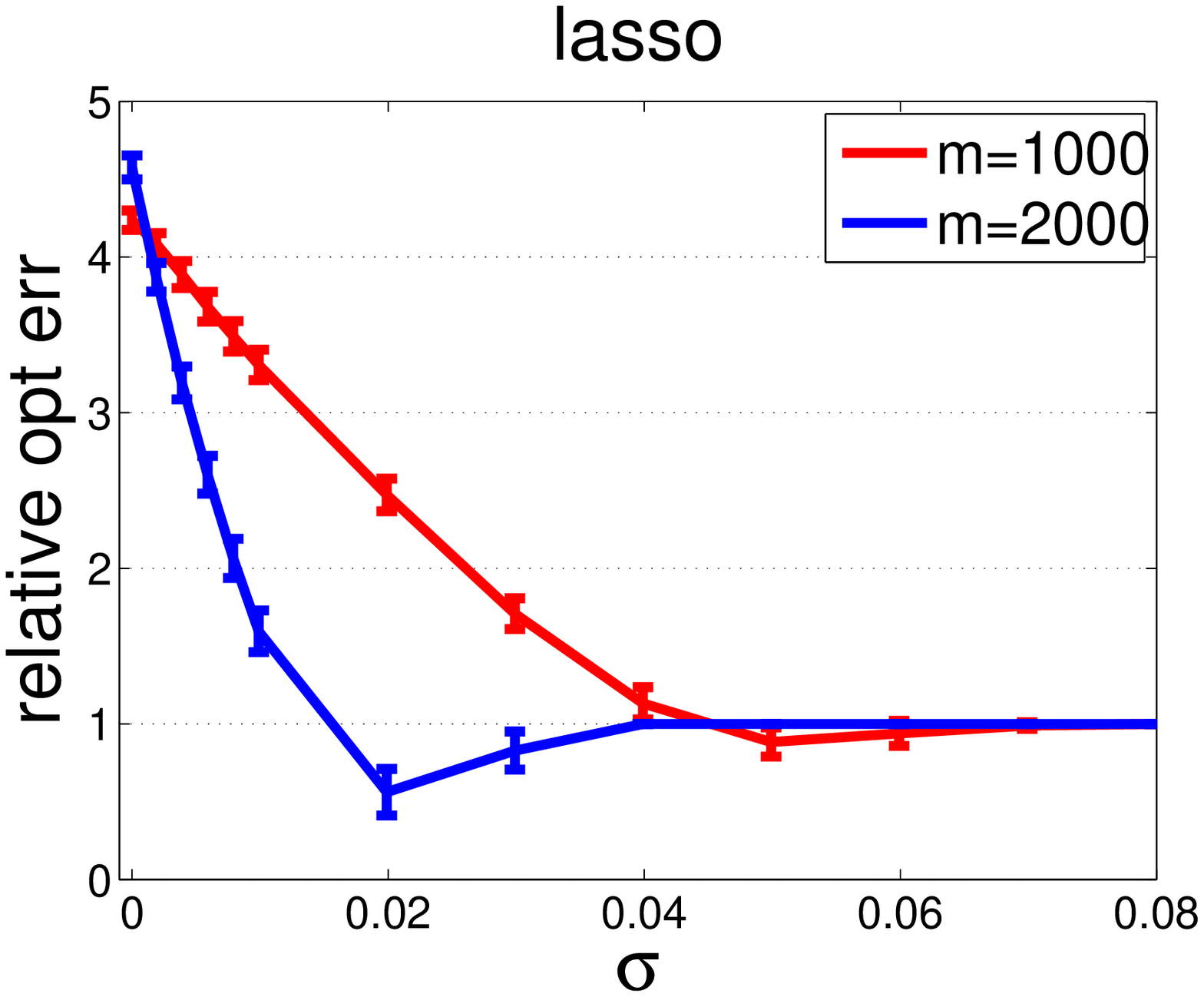}\hspace*{-0.1in}
 \includegraphics[width=0.3\textwidth]{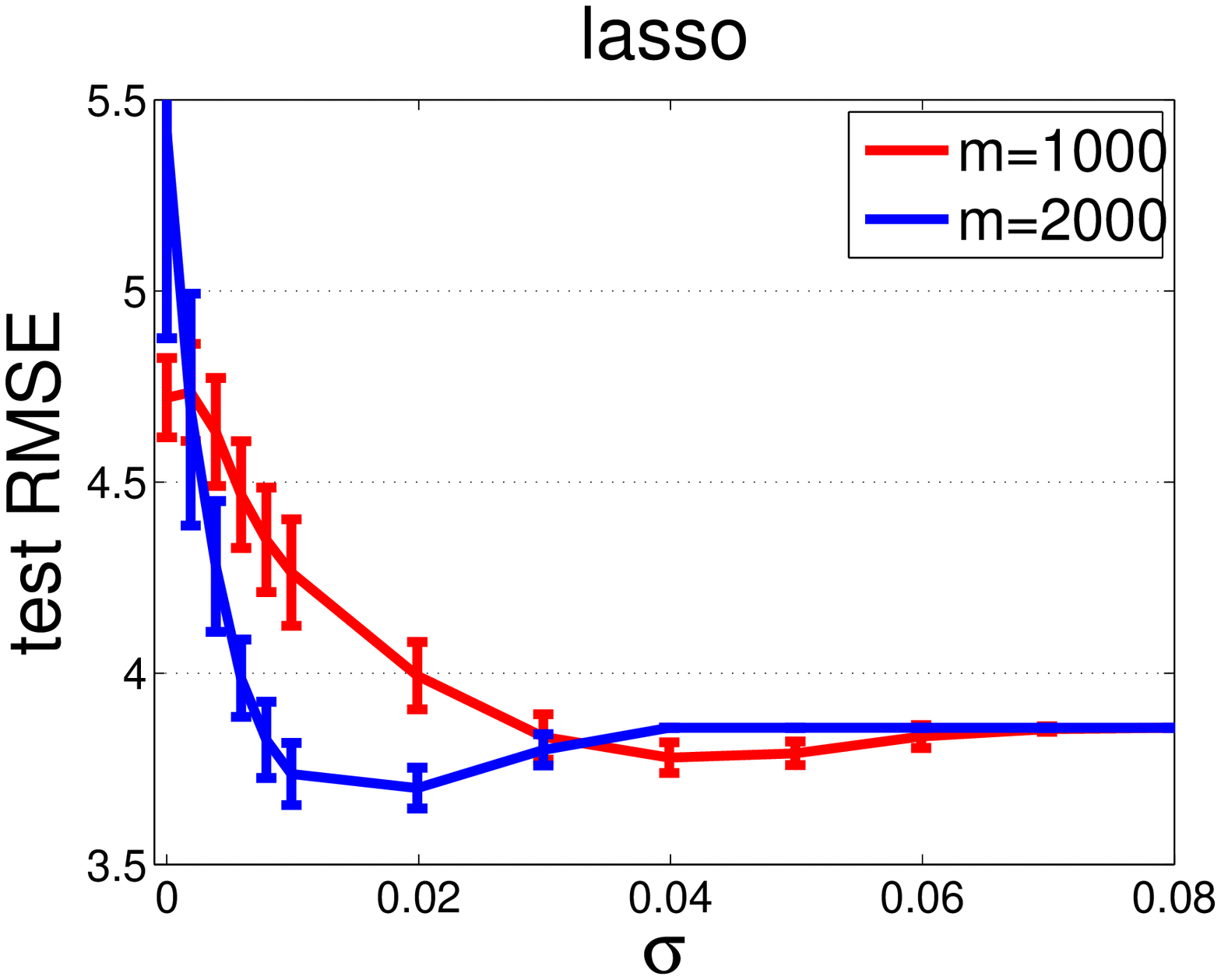}\hspace*{-0.1in}
 \vspace*{-0.1in}
\caption{Optimization or Regression error of lasso under different settings on the E2006-tfidf.}
\label{fig:E2006}
\vspace*{-0.2in}
\end{figure}

\section{Numerical Experiments}\label{sec:exp}
\vspace*{-0.1in}
In this section, we present some numerical experiments to complement  the theoretical results. We conduct experiments on two datasets, a synthetic dataset and a real dataset. The synthetic data is generated similar to previous studies on sparse signal recovery~\cite{DBLP:journals/siamjo/Xiao013}.  
In particular, we generate a random matrix $X\in\R^{n\times d}$ with $n = 10^4$ and $d = 10^5$. The entries of the matrix $X$ are generated independently with the uniform distribution over the interval $[-1,+1]$. A sparse vector $\u_* \in \R^d$ is generated with the same distribution at $100$ randomly chosen coordinates. The noise $\xi\in\R^n$ is a dense vector with independent random entries with the uniform distribution over the interval $[-\sigma, \sigma]$, where $\sigma$ is the noise magnitude and is set to $0.1$. We scale the data matrix $X$ such that all entries have a variance of $1/n$ and scale the noise vector $\xi$  accordingly.  Finally the vector $\y$ was obtained as $\y = X\u_* + \xi$. For elastic net on the synthetic data, we try two different values of $\lambda$, $10^{-8}$ and $10^{-5}$. The value of $\tau$ is set to $10^{-5}$ for both elastic net and lasso. Note that these values are not intended to optimize the performance of elastic net and lasso on the synthetic data. The real data used in the experiment is E2006-tfidf dataset. We use the version available on libsvm website~\footnote{\url{http://www.csie.ntu.edu.tw/~cjlin/libsvmtools/datasets/}}. There are a total of $n=16,087$ training instances and $d=150,360$ features and $3308$ testing instances.  We normalize the training data such that each dimension has mean zero and variance $1/n$. The testing data is normalized using the statistics computed on the training data. For JL transform, we use the random hashing.  

The experimental results on the synthetic data under different settings are shown in Figure~\ref{fig:syn}. In the left plot, we compare the optimization error for elastic net with  $\lambda=10^{-8}$ and  two different values of $m$, i.e., $m=1000$ and $m=2000$. The horizontal axis is the value of $\sigma$, the added regularization parameter. We can observe that adding a slightly larger additional $\ell_1$ norm to the compressed data problem indeed reduces the optimization error. When the value of $\sigma$ is larger than some threshold, the error will increase, which is consistent with our theoretical results. In particular, we can see that the threshold value for $m=2000$ is smaller than that for $m=1000$. In the middle plot, we compare the optimization error for elastic net with $m=1000$ and two different values of the regularization parameter $\lambda$. Similar trends of the optimization error versus $\sigma$ are also observed. In addition, it is interesting to see that the optimization error for $\lambda=10^{-8}$ is less than that for $\lambda=10^{-5}$, which seems to contradict to the theoretical results at the first glance due to the explicit inverse dependence on $\lambda$. However, the optimization error also depends on $\|\e\|_2$, which measures the empirical error of the corresponding optimal model. We find that with $\lambda=10^{-8}$ we have a smaller $\|\e\|_2=0.95$ compared to $1.34$ with $\lambda=10^{-5}$, which explains the result in the middle plot. For the right plot, we repeat the same experiments for lasso as in the left plot for elastic net, and observe  similar results. 

The experimental results on E2006-tfidf dataset for lasso are shown in Figure~\ref{fig:E2006}. In the left plot, we show the root mean square error (RMSE) on the testing data of different models learned from the original data with different values of $\tau$. In the middle and right plots, we fix the value of $\tau=10^{-4}$ and increase the value of $\sigma$ and plot the relative optimization error and the RMSE on the testing data. Again, the empirical results are consistent with the theoretical results and verify  that with JL transforms a larger $\ell_1$ regularizer yields a better performance.

\section{Conclusions}
\vspace*{-0.1in}
In this paper, we have considered a fast approximation method for sparse least-squares regression by exploiting the JL transform. We propose a slightly different formulation on the compressed data and interpret it from a geometric viewpoint. We also establish the theoretical guarantees on the optimization error of the obtained solution for elastic net,  lasso and Dantzig selector on the compressed data. The theoretical results are also validated by numerical experiments on a synthetic dataset and a real dataset.

{
\bibliographystyle{abbrv}
\bibliography{all}}

\appendix
\section{Proofs of main theorems}\label{sec:ana}
\subsection{Proof of Theorem 2}
Recall  the definitions
\begin{align}\label{eqn:q}
\displaystyle\q  = \frac{1}{n}X^{\top}(A^{\top}A -I )\e, \quad\quad \e=X\w_* - \y
\end{align}
First, we note that 
\begin{align*}
\wh_* &= \arg\min_{\w\in\R^d}\frac{1}{2n}\|\Xh\w - \yh\|_2^2 + \frac{\lambda}{2}\|\w\|_2^2 + (\tau + \sigma) \|\w\|_1\\
&=\arg\min_{\w\in\R^d}\underbrace{\frac{1}{2n}\left(\w^{\top}\Xh^{\top}\Xh\w - 2\w^{\top}\Xh^{\top}\y \right)+ \frac{\lambda}{2}\|\w\|_2^2 + (\tau + \sigma) \|\w\|_1}\limits_{F(\w)}
\end{align*}
and 
\begin{align*}
\w_* = \arg\min_{\w\in\R^d}\frac{1}{2n}\|X\w - \y\|_2^2 + \frac{\lambda}{2}\|\w\|_2^2 + \tau \|\w\|_1
\end{align*}
By optimality of $\wh_*$ and the strong convexity of $F(\w)$, for any $g\in\partial\|\w_*\|_1$ we have
\begin{align}\label{eqn:fobj}
0\geq F(\wh_*) - F(\w_*)\geq& (\wh_* - \w_*)^{\top}\left(\frac{1}{n}\Xh^{\top}\Xh\w_*  - \frac{1}{n}\Xh^{\top}\yh + \lambda \w_*\right)  + (\tau + \sigma)(\wh_* - \w_*)^{\top}g \notag\\
&+ \frac{\lambda}{2}\|\wh_* - \w_*\|_2^2
\end{align}
By the optimality condition of  $\w_*$, there exists $h\in\partial \|\w_*\|_1$ such that  
\begin{align}\label{eqn:opt}
\frac{1}{n}X^{\top}X\w_* -\frac{1}{n}X^{\top} \y+ \lambda \w_* + \tau h=0 
\end{align}
By utilizing the above equation in~(\ref{eqn:fobj}), we have
\begin{align}\label{eqn:key}
0\geq&   (\wh_* - \w_*)^{\top}\q + (\wh_* - \w_*)^{\top}\left[(\tau + \sigma)g - \tau h\right] + \frac{\lambda}{2}\|\wh_* - \w_*\|_2^2
\end{align}
Let $\S$ denote the support set of $\w_*$ and $\S_c$ denote its complement set.  Since $g$ could be any sub-gradient of $\|\w\|_1$ at $\w_*$, we define $g$ as $g_i = \left\{\begin{array}{lc} h_i,& i\in\S\\sign(\widehat w_{*i}),&i\in\S_c\end{array}\right.$.  Then we have 
\begin{align*}
(\wh_* - \w_*)^{\top}&\left[(\tau + \sigma)g - \tau h\right] = \sum_{i\in\S}(\widehat w_{*i} - w_{*i} )(\sigma h_i) + \sum_{i\in\S^c}(\widehat w_{*i} - w_{*i} )(\sigma sign(\widehat w_{*i}) + \tau(sign(\widehat w_{*i}) - h_i))\\
&\geq -\sigma\|[\wh_* - \w_*]_{\S}\|_1 + \sum_{i\in\S_c}\sigma sign(\widehat w_{*i})\widehat w_{*i}  + \sum_{i\in\S_c}\tau(sign(\widehat w_{*i}) - h_i)\widehat w_{*i}\\
&\geq -\sigma\|[\wh_* - \w_*]_{\S}\|_1  + \sigma \|[\wh_*]_{\S_c}\|_1
\end{align*}
where the last inequality uses $|h_i|\leq 1$ and $\sum_{i\in\S_c}(sign(\widehat w_{*i}) - h_i)\widehat w_{*i}\geq 0$. 
Combining the above inequality with~(\ref{eqn:key}), we have
\begin{align*}
0\geq&  -\|\wh_* - \w_*\|_1\|\q\|_\infty- \sigma\|[\wh_* - \w_*]_\S\|_1  +\sigma\|[\wh_*]_{\S_c}\|_1 + \frac{\lambda}{2}\|\wh_* - \w_*\|_2^2
\end{align*}
By splitting $\|\wh_* - \w_*\|_1 = \|[\wh_* - \w_*]_\S\|_1 + \|[\wh_* - \w_*]_{\S_c}\|_1$ and reorganizing the above inequality we have
\begin{align*}
\frac{\lambda}{2}\|\wh_* - \w_*\|_2^2 + (\sigma - \|\q\|_\infty)\|[\wh_*]_{\S^c}\|_1\leq (\sigma + \|\q\|_\infty)\|[\wh_* - \w_*]_\S\|_1
\end{align*}
If $\sigma\geq 2\|\q\|_\infty$, then we have
\begin{align}
\frac{\lambda }{2}\|\wh_* - \w_*\|_2^2&\leq \frac{3\sigma}{2}\|[\wh_* - \w_*]_\S\|_1\label{eqn:b1}\\
\|[\wh_*]_{\S^c}\|_1&\leq 3\|[\wh_* - \w_*]_\S\|_1\label{eqn:b2}
\end{align}
Note that the inequality~(\ref{eqn:b2}) hold regardless the value of $\lambda$. 
Since
\begin{align*}
\|[\wh_*-\w_*]_\S\|_1\leq \sqrt{s}\|[\wh_*-\w_*]_\S\|_2, \text{ and  }\|\wh_* - \w_*\|_2\geq \max(\|[\wh_*-\w_*]_\S\|_2,  \|[\wh_*]_{\S^c}\|_2),
\end{align*}
by combining the above inequalities with~(\ref{eqn:b1}), we can get
\begin{align*}
 \|\wh_* - \w_*\|_2\leq \frac{3\sigma}{\lambda}\sqrt{s},  \quad  \|[\wh_* - \w_*]_\S\|_1\leq \frac{3\sigma}{\lambda}s
\end{align*}
and
\begin{align*}
\|\wh_* - \w_*\|_1\leq  \|[\wh_*]_{\S^c}\|_1 + \|[\wh_* - \w_*]_\S\|_1\leq3\|[\wh_* - \w_*]_\S\|_1 +\|[\wh_* - \w_*]_\S\|_1\leq  \frac{12\sigma}{\lambda }s
\end{align*} 
We can then complete the proof of Theorem~2 by noting the upper bound of $\|\q\|_\infty$  in Lemma 2 and by setting $\sigma$ according to the Theorem. 

\subsection{Proof of Theorem 3}
When $\lambda=0$, the reduced problem becomes 
\begin{align}
\wh_* &=\arg\min_{\w\in\R^d}\underbrace{\frac{1}{2n}\|\Xh\w - \yh\|_2^2  +(\tau+\sigma)\|\w\|_1}\limits_{F(\w)}
\end{align}
From the proof of Theorem 2, we have 
\begin{align*}
\|[\wh_*]_{\S^c}\|_1&\leq 3\|[\wh_* - \w_*]_\S\|_1, \quad \text{and}\quad \frac{\|\wh_* - \w_*\|_1}{\|\wh_* - \w_*\|_2} = \frac{4\|[\wh_* - \w_*]_{\S}\|_1}{\|\wh_* - \w_*\|_2} \leq 4\sqrt{s}
\end{align*}

\setcounter{lemma}{3}
Then we can have the following lemma, whose proof of the lemma is deferred to next section.  
\begin{lemma}
If $X$ satisfies the restricted eigen-value condition at sparsity level $16s$, then 
\begin{align*}
\phi_{\min}(16s)\|\wh_* - \w_*\|^2_2\leq (\wh_*-\w_*)^{\top}X^{\top}X(\wh_* - \w_*)\leq 4\phi_{\max}(16s)\|\wh_* - \w_*\|^2_2
\end{align*}
\end{lemma}
Then we proceed our proof as follows. Since $\w_*$ optimizes the original problem, we have for any $g\in\partial \|\wh_*\|_1$
\begin{align*}
0\geq (\w_* - \wh_*)^{\top}\left(\frac{1}{n}X^{\top}X\wh_* - \frac{1}{n}X^{\top}\y \right) + \tau (\w_* - \wh_*)^{\top}g + \frac{1}{2n}(\w_* - \wh_*)^{\top}X^{\top}X(\w_* - \wh_*)
\end{align*}
Since $\wh_*$ optimizes $F(\w)$, there exists $h\in\partial \|\wh_*\|_1$, we have
\begin{align*}
0\geq (\wh_* - \w_*)^{\top}\left(\frac{1}{n}\Xh^{\top}\Xh\wh_* - \frac{1}{n}\Xh^{\top}\yh\right) + (\tau + \sigma)(\wh_* - \w_*)^{\top}h 
\end{align*}
Combining the two inequalities above we have
\begin{align*}
0\geq& (\w_* - \wh_*)^{\top}\left(\frac{1}{n}X^{\top}X\wh_* - \frac{1}{n}X^{\top}\y - \frac{1}{n}\Xh^{\top}\Xh\wh_* + \frac{1}{n}\Xh^{\top}\yh\right) + (\wh_* - \w_*)^{\top}(\tau h + \sigma h - \tau g)\\
& + \frac{1}{2n}(\w_* - \wh_*)^{\top}X^{\top}X(\w_* - \wh_*)\\
&=(\w_* - \wh_*)^{\top}\left(\frac{1}{n}X^{\top}X\w_* - \frac{1}{n}X^{\top}\y - \frac{1}{n}\Xh^{\top}\Xh\w_* + \frac{1}{n}\Xh^{\top}\yh\right) + (\wh_* - \w_*)^{\top}(\tau h + \sigma h - \tau g)\\
& + \frac{1}{2n}(\w_* - \wh_*)^{\top}X^{\top}X(\w_* - \wh_*) + (\w_* - \wh_*)^{\top}\left(\frac{1}{n}X^{\top}X(\wh_* - \w_*)   - \frac{1}{n}\Xh^{\top}\Xh(\wh_*-\w_*)  \right)\\
&=(\w_* - \wh_*)^{\top}\left(\frac{1}{n}X^{\top}X\w_* - \frac{1}{n}X^{\top}\y - \frac{1}{n}\Xh^{\top}\Xh\w_* + \frac{1}{n}\Xh^{\top}\yh\right) + (\wh_* - \w_*)^{\top}(\tau h + \sigma h - \tau g)\\
& + \frac{1}{2n}(\w_* - \wh_*)^{\top}X^{\top}X(\w_* - \wh_*) + (\w_* - \wh_*)^{\top}\left(\frac{1}{n}X^{\top}X - \frac{1}{n}\Xh^{\top}\Xh \right)(\wh_* - \w_*)  
\end{align*}
By setting  $g_i= h_i, i\in \S$ and following the same analysis as in the Proof of Theorem~2, we have
\begin{align*}
(\wh_* - \w_*)^{\top}(\tau h + \sigma h - \tau g) \geq - \sigma\|[\wh_* - \w_*]_\S\|_1+ \sigma\|[\wh_*]_{\S^c}\|_1
\end{align*}
As a result, 
\begin{align*}
0\geq - \|\wh_* - \w_*\|_1\|\q\|_\infty - \sigma\|[\wh_* - \w_*]_\S\|_1 + \sigma\|[\wh_*]_{\S^c}\|_1 + \frac{\phi_{\min}(16s)}{2}\|\wh_* - \w_*\|_2^2 - \rho(16s)\|\wh_* - \w_*\|_2^2
\end{align*}
Then if $\sigma\geq 2\|\q\|_\infty$, we arrive at the same conclusion with $\lambda$ replaced by $\phi_{\min}(16s)-2\rho(16s)$ assuming $\phi_{\min}(16s)\geq 2\rho(16s)$.

\subsection{Proof of Theorem 4}
Let $\bdelta = \wh_* - \w_*$. First we show that 
\[
\|[\bdelta]_{\S_c}\|_1\leq \|[\bdelta]_\S\|_1
\]
This is because 
\begin{align*}
\|\w_*\|_1 - \|[\bdelta]_\S\| + \|[\bdelta]_{\S_c}\|_1  \leq \|\w_* + \delta \|_1=\|\wh_*\|_1\leq \|\w_*\|_1 
\end{align*}
Therefore  $\|[\bdelta]_{\S_c}\|_1\leq \|[\bdelta]_\S\|_1$, and we have 
\begin{align*}
\|[\wh_*]_{\S_c}\|_1&\leq \|[\wh_* - \w_*]_\S\|_1, \quad \text{and}\quad \frac{\|\wh_* - \w_*\|_1}{\|\wh_* - \w_*\|_2} = \frac{2\|[\wh_* - \w_*]_{\S}\|_1}{\|\wh_* - \w_*\|_2} \leq 2\sqrt{s}
\end{align*}
Similarly, we have the following lemma. 
\begin{lemma}
If $X$ satisfies the restricted eigen-value condition at sparsity level $4s$, then 
\begin{align*}
\phi_{\min}(4s)\|\wh_* - \w_*\|^2_2\leq \frac{1}{n}(\wh_*-\w_*)^{\top}X^{\top}X(\wh_* - \w_*)\leq 4\phi_{\max}(4s)\|\wh_* - \w_*\|^2_2
\end{align*}
\end{lemma}
We continue the proof as follows:  
\begin{align*}
\frac{1}{n}\|X\bdelta\|^2_2 \leq  \frac{1}{n}\|\Xh\bdelta\|^2_2+ \frac{1}{n}\left|\bdelta^{\top}(X^{\top}X- \Xh^{\top}\Xh)\bdelta\right|
\end{align*}
Since 
\begin{align*}
\frac{1}{n}\bdelta^{\top}\Xh^{\top}\Xh\bdelta &\leq \|\bdelta\|_1 \frac{1}{n}\left\|\Xh^{\top}\Xh\bdelta\right\|_\infty\\
&\leq \|\bdelta\|_1 \frac{1}{n}\left\|\Xh^{\top}(\Xh\wh_*-\yh) - \Xh^{\top}(\Xh\w_*-\yh)\right\|_\infty\\
&\leq \|\bdelta\|_12(\tau + \sigma)
\end{align*}
Then we have
\begin{align}
\phi_{\min}(4s)\|\wh_* - \w_*\|^2_2&\leq2(\tau + \sigma) \|\wh_* - \w_*\|_1 + \rho(4s)\|\wh_* - \w_*\|^2_2\notag\\
&\leq 4(\tau + \sigma)\|[\wh_* - \w_*]_\S\|_1 +  \rho(4s)\|\wh_* - \w_*\|^2_2\label{eqn:l}
\end{align}
Then we have
\begin{align*}
\|\wh_* - \w_*\|_2\leq \frac{4(\tau + \sigma)\sqrt{s}}{\phi_{\min}(4s) - \rho(4s)}, \quad \|\wh_* - \w_*\|_1\leq \frac{4(\tau + \sigma)s}{\phi_{\min}(4s) - \rho(4s)}
\end{align*}
We  then complete the proof of Theorem~4 by noting the upper bound of $\|\q\|_\infty$ and by setting $\sigma$ according to the Theorem.

\section{Proofs of Lemmas}
\subsection{Proof of Lemma 2}
The proof of Lemma 2 follows that of Theorem 6  in~\cite{DBLP:journals/corr/Yang0JZ15}. For completeness, we present the proof here. 
Since $X=(\bar\x_1,\ldots, \bar\x_d)$, 
\begin{align*}
\|\q\|_\infty = \max_{1\leq j\leq d}\frac{1}{n}|\bar\x_j^{\top}(I - A^{\top}A)\e|
\end{align*}
We first bound for individual $j$ and then apply the union bound. Let $\xt_i$ and $\et_*$ be normalized version of $\bar\x_i$ and $\e$, i.e., $\xt_i = \bar\x_i/\|\bar\x_i\|_2$ and $\et = \e/\|\e\|_2$. Let $\epsilon\triangleq=c\sqrt{\frac{\log(1/\delta)}{m}}$.  Since $A$ obeys the JL lemma, therefore with a probability $1-\delta$ we have
\[
\left|\|A\x\|_2^2 - \|\x\|_2^2\right|\leq \epsilon\|\x\|_2^2
\]
Then with a probability $1-\delta$, 
\begin{align*}
 \xt_j^{\top}A^{\top}A\et - \xt_j^{\top}\et &= \frac{\|A(\xt_j+\et)\|_2^2 - \|A(\xt_j-\et)\|_2^2}{4} - \xt_i^{\top}\et\\
 &\leq \frac{(1+\epsilon) \|\xt_j+\et\|_2^2 + (1-\epsilon) \|\xt_j- \et\|_2^2}{4}- \xt_i^{\top}\et\\
 &\leq \frac{\epsilon}{2}(\|\xt_j\|_2^2 + \|\et\|_2^2)\leq \epsilon
\end{align*}
Similarly with a probability $1- \delta$, 
\begin{align*}
 \xt_j^{\top}A^{\top}A\et - \xt_j^{\top}\et &= \frac{\|A(\xt_j+\et)\|_2^2 - \|A(\xt_j-\et)\|_2^2}{4} - \xt_j^{\top}\et\geq -\frac{\epsilon}{2}(\|\xt_j\|_2^2 + \|\et\|_2^2)\geq -\epsilon
\end{align*}
Therefore with a probability $1 - 2\delta$, we have
\begin{align*}
 &|\bar\x_j^{\top}A^{\top}A\e - \bar\x_i^{\top}\e|\leq \|\bar\x_j\|_2\|\e\|_2 |\xt_j^{\top}A^{\top}A\et - \xt^{\top}\et|\leq \|\bar\x_j\|_2\|\e\|_2 \epsilon
\end{align*}
Then applying union bound, we complete the proof. 

\subsection{Proof of Lemma 3}
The proof of Lemma 3 follows the analysis in~\cite{DBLP:journals/corr/Yang0JZ15}. For completeness, we present the proof here. 
Define $\mathcal S_{d,s}$ and $\mathcal K_{d,s}$: 
\begin{align*}
\mathcal S_{d,s}=\{\u\in\R^d: \|\u\|_2\leq 1, \|\u\|_0\leq s\}, \quad \mathcal K_{d,s}=\{\u\in\R^d: \|\u\|_2\leq 1, \|\u\|_1\leq \sqrt{s}\}
\end{align*}
Due to $conv(\mathcal S_{d,s})\subseteq\mathcal K_{d,s}\subseteq 2conv(\mathcal S_{d,s})$~\cite{DBLP:journals/corr/abs-1109-4299}, 
for any $\u\in\mathcal K_{d,s}$, we can write it as $\u=2\sum_i\lambda_i\v_i$ where $\v_i\in\mathcal S_{d,s}$, $\sum_i\lambda_i=1$ and $\lambda_i\geq 0$, then we have 
\begin{align*}
&|\u^{\top}(X^{\top}X - \Xh^{\top}\Xh)\u|=|(X\u)^{\top}(I - A^{\top}A)(X\u)|\\
&\leq 4\left|\left(X\sum_{i}\lambda_i\v_i\right)^{\top}(I - A^{\top}A)\left(X\sum_i\lambda_i\v_i\right)\right|\leq 4\sum_{ij}\lambda_i\lambda_j|(X\v_i)^{\top}(I - A^{\top}A)(X\v_j)|\\
&\leq 4\max_{\u_1, \u_2\in \mathcal S_{d,s}}|(X\u_1)^{\top}(I - A^{\top}A)(X\u_2)|\sum_{ij}\lambda_i\lambda_j= 4\max_{\u_1, \u_2\in \mathcal S_{d,s}}|(X\u_1)^{\top}(I - A^{\top}A)(X\u_2)
\end{align*}
Therefore 
\begin{align}\label{eqn:k}
&\max_{\u\in\mathcal K_{d,s}}|(X\u)^{\top}(I - A^{\top}A)(X\u)|\leq 4\max_{\u_1, \u_2\in \mathcal S_{d,s}}|(X\u_1)^{\top}(I - A^{\top}A)(X\u_2)
\end{align}
Following the Proof of Lemma 2, for any fixed $\u_1,\u_2\in\mathcal S_{d,s}$, with a probability $1-2\delta$ we  have 
\begin{align*}
&\frac{1}{n}|(X\u_1)^{\top}(I - A^{\top}A)(X\u_2)|\leq \frac{1}{n}\|X\u_1\|_2\|X\u_2\|_2\epsilon\leq \phi_{\max}(s)c\sqrt{\frac{\log(1/\delta)}{m}}
\end{align*}
where we use the restricted eigen-value condition
\[
\max_{\u\in \mathcal S_{d,s}}\frac{\|X\u\|_2}{\sqrt{n}}= \sqrt{\phi_{\max}(s)}
\]
To prove the bound for  all $\u_1, \u_2\in\mathcal S_{d,s}$, we consider the $\epsilon$ proper-net of $\mathcal S_{d,s}$~\cite{DBLP:journals/corr/abs-1109-4299} denoted by $\mathcal S_{d,s}(\epsilon)$.  Lemma 3.3 in~\cite{DBLP:journals/corr/abs-1109-4299} shows that the entropy of $\mathcal S_{d,s}$, i.e., the cardinality of $\mathcal S_{d,s}(\epsilon)$ denoted $N(\mathcal S_{d,s},\epsilon)$ is bounded by
\[
\log N(\mathcal S_{d,s},\epsilon) \leq s\log\left(\frac{9d}{\epsilon s}\right)
\]
Then by using the union bound, we have with a probability $1-2\delta$, we have
\begin{align}\label{eqn:dis}
&\max_{\u_1\in\mathcal S_{d,s}(\epsilon)\atop \u_2\in\mathcal S_{d,s}(\epsilon)}\frac{1}{n}|(X\u_1)^{\top}(I - A^{\top}A)(X\u_2)|\leq \phi_{\max}(s)c\sqrt{\frac{\log(N^2(\mathcal S_{d,s},\epsilon)/\delta)}{m}}\nonumber\\
&\leq \phi_{\max}(s)c\sqrt{\frac{\log(1/\delta) + 2s\log(9d/\epsilon s)}{m}}
\end{align}
To proceed the proof, we need the following lemma. 
\begin{lemma} \label{lemma:discrete-bound}
Let \begin{align*}
\Er_s(\u_2) = \max_{\u_1\in\Se_{d,s}} |\u_1^{\top}U\u_2|\\
\Er_s(\u_2,\epsilon) = \max_{\u_1\in\Se_{d,s}(\epsilon)} |\u_1^{\top}U\u_2|
\end{align*}
For  $\epsilon \in (0, 1/\sqrt{2})$, we have
\begin{align*}
\Er_s(\u_2)\leq \left(\frac{1}{1 - \sqrt{2}\epsilon}\right)\Er_s(\u_2,\epsilon)\end{align*}
\end{lemma}
\begin{proof} Let $U = \frac{1}{n}X^{\top}(I - A^{\top}A)X$.  Following  Lemma 9.2 of~\cite{koltchinskii2011oracle}, for any $\u, \u' \in \Se_{d,s}$, we can always find two vectors $\v$, $\v'$ such that
\[
\u-\u'=\v-\v', \ \|\v\|_0 \leq s, \ \|\v'\|_0 \leq s, \ \v^\top \v'=0.
\]
Thus
\[
\begin{split}
& |\langle \u - \u', U\u_2 \rangle| \leq  |\langle \v, U\u_2 \rangle| + |\langle -\v', U\u_2 \rangle| \\
= & \|\v\|_2 \left|\left \langle \frac{\v}{\|\v\|_2}, U\u_2 \right \rangle\right| + \|\v'\|_2 \left|\left \langle \frac{-\v'}{\|\v'\|_2}, U\u_2 \right \rangle\right| \\
\leq & (\|\v\|_2+ \|\v'\|_2) \Er_s(\u_2)  \leq \Er_s(\u_2)  \sqrt{2} \sqrt{\|\v\|_2^2+ \|\v'\|_2^2}\\
=& \Er_s(\u_2)  \sqrt{2} \|\v-\v'\|_2= \Er_s(\u_2)  \sqrt{2} \|\v-\v'\|_2= \Er_s(\u_2)  \sqrt{2} \|\u-\u'\|_2.
\end{split}
\]

Then, we have
\[
\begin{split}
& \Er_s(\u_2) = \max\limits_{\u \in \Se_{d,s}} |\u^{\top}U\u_2|\leq  \max\limits_{\u \in \Se_{d,s}(\epsilon)} |\u^{\top}U\u_2| + \sup_{\u \in \Se_{d,s}\atop \u' \in \Se_{d,s}(\epsilon), \|\u-\u'\|_2 \leq \epsilon} \langle \u-\u', U\u_2 \rangle \\
\leq & \Er_s(\u_2, \epsilon) +  \sqrt{2} \epsilon \Er_s(\u_2)
\end{split}
\]
which implies
\[
\Er_s(\u_2) \leq \frac{\Er_s(\u_2, \epsilon)}{1-\sqrt{2} \epsilon}.
\]
\end{proof}

\begin{lemma} \label{lemma:discrete-bound-2}
Let \begin{align*}
\Er_s(\epsilon) = \max_{\u_2\in\Se_{d,s}}\Er_s(\u_2,\epsilon)=\max_{\u_1\in\Se_{d,s}\atop \u_2\in\Se_{d,s}(\epsilon)} |\u_1^{\top}U\u_2|\\
\Er_s(\epsilon,\epsilon)=\max_{\u_2\in\Se_{d,s}(\epsilon)}\Er_s(\u_2,\epsilon) = \max_{\u_1,\u_2\in\Se_{d,s}(\epsilon)} |\u_1^{\top}U\u_2|
\end{align*}
For  $\epsilon \in (0, 1/\sqrt{2})$, we have
\begin{align*}
\Er_s(\epsilon)\leq \left(\frac{1}{1 - \sqrt{2}\epsilon}\right)\Er_s(\epsilon,\epsilon)\end{align*}
\end{lemma}
The proof the above lemma follows the same analysis as that of Lemma 6. 
By combining Lemma 6 and Lemma 7, we have
\begin{align*}
 &\max_{\u_2\in\Se_{d,s}}\Er_s(\u_2)\leq \frac{\max_{\u_2\in\Se_{d,s}}\Er_s(\u_2, \epsilon)}{1-\sqrt{2} \epsilon}= \frac{1}{1-\sqrt{2}\epsilon}\Er_s(\epsilon)\leq \left(\frac{1}{1-\sqrt{2}\epsilon}\right)^2\Er_s(\epsilon,\epsilon)\\
& =  \left(\frac{1}{1-\sqrt{2}\epsilon}\right)^2 \max_{\u_1,\u_2\in\Se_{d,s}(\epsilon)} |\u_1^{\top}U\u_2|
\end{align*}
By combing the above inequality with inequality~\ref{eqn:k} and~(\ref{eqn:dis}), we have
\[
\rho_s\leq 4\max_{\u_2\in\Se_{d,s}}\Er_s(\u_2)\leq 4\left(\frac{1}{1-\sqrt{2}\epsilon}\right)^2 \phi_{\max}(s)c\sqrt{\frac{\log(1/\delta) + 2s\log(9d/\epsilon s)}{m}}
\]
If we set $\epsilon=1/(2\sqrt{2})$, we can complete the proof. 

\subsection{Proof of Lemma 4}
Since 
\[
 \frac{\|\wh_* - \w_*\|_1}{\|\wh_* - \w_*\|_2}  \leq 4\sqrt{s}=\sqrt{16s},
\]
Therefore $ \frac{\|\wh_* - \w_*\|_1}{\|\wh_* - \w_*\|_2}\in\mathcal K_{d,16s}$. 
The left inequality follows the restricted eigen-value condition and $conv(\mathcal S_{d,s})\subseteq\mathcal K_{d,s}$. For the right inequality, we note that $\mathcal K_{d,s}\subseteq 2conv(\mathcal S_{d,s})$, hence for any $\u\in\mathcal K_{d,s}$, we can write $\u=2\sum_{i}\lambda_i\v_i$ with $\sum_{i}\lambda_i=1$, $\lambda_i\geq0$, and $\v_i\in\mathcal S_{d,s}$. 
\[
\frac{1}{n}\u^{\top}X^{\top}X\u=f(\u)=f(2\sum_i\lambda_i\v_i)\leq\sum_i\lambda_i f(2\v_i)\leq \frac{1}{n} \sum_i\lambda_i 4\v_i^{\top}X^{\top}X\v_i\leq 4 \phi_{\max}(s)
\]
Therefore
\[
\frac{1}{n}(\wh_* - \w_*)^{\top}X^{\top}X(\wh_* - \w_*)\leq  4\phi_{\max}(16s)\|\wh_* - \w_*\|^2_2
\]

\end{document}